\documentclass[12pt]{amsart}
\usepackage[a4paper,margin=3cm]{geometry}

\usepackage{amsthm} 
\usepackage{amsmath} 
\usepackage{amssymb} 

\usepackage{microtype} 
\usepackage{pinlabel}
\usepackage{color}
\usepackage{import}

\usepackage[pagebackref]{hyperref}

\renewcommand*{\backref}[1]{}
\renewcommand*{\backrefalt}[4]{
  \ifcase #1 
  [No citations.]
  \or [#2]
  \else [#2]
  \fi }

\let\originalleft\left
\let\originalright\right
\renewcommand{\left}{\mathopen{}\mathclose\bgroup\originalleft}
\renewcommand{\right}{\aftergroup\egroup\originalright}

\newcommand{\calB}{\mathcal{B}}
\newcommand{\calC}{\mathcal{C}}

\newcommand{\calR}{\mathcal{R}}
\newcommand{\calS}{\mathcal{S}}
\newcommand{\calT}{\mathcal{T}}

\newcommand{\RR}{\mathbb{R}}

\newcommand{\homeo}{\mathrel{\cong}} 

\newcommand{\cross}{\times}

\newcommand{\bdy}{\partial} 

\newcommand{\pt}{\rm pt} 

\numberwithin{equation}{section} 

\makeatletter
\let\c@figure\c@equation
\makeatother
\numberwithin{figure}{section} 

\theoremstyle{plain}
\newtheorem{theorem}[equation]{Theorem}
\newtheorem{corollary}[equation]{Corollary}
\newtheorem{lemma}[equation]{Lemma}

\theoremstyle{definition}
\newtheorem{definition}[equation]{Definition}

\newtheorem{exercise}[equation]{Exercise}

\newtheorem*{question*}{Question}

\theoremstyle{remark}
\newtheorem{remark}[equation]{Remark}
\newtheorem*{remark*}{Remark}

\newtheorem*{case*}{Case}
\newtheorem*{step*}{Step}

\newtheorem*{claim*}{Claim}

\newcommand{\refsec}[1]{Section~\ref{Sec:#1}}

\newcommand{\refthm}[1]{Theorem~\ref{Thm:#1}}
\newcommand{\refcor}[1]{Corollary~\ref{Cor:#1}}
\newcommand{\reflem}[1]{Lemma~\ref{Lem:#1}}

\newcommand{\reffig}[1]{Figure~\ref{Fig:#1}}

\theoremstyle{plain}
\newtheorem{XXXtheoremQED}[equation]{Theorem} 
\newenvironment{theoremQED}                   
  {\pushQED{\qed}\begin{XXXtheoremQED}}
  {\popQED\end{XXXtheoremQED}}

\newcommand{\fakeenv}{} 

\newenvironment{restate}[2]  
{ 
 \renewcommand{\fakeenv}{#2} 
 \theoremstyle{plain} 
 \newtheorem*{\fakeenv}{#1~\ref{#2}} 
 \begin{\fakeenv}
}
{
 \end{\fakeenv}
}

\newenvironment{restated}[2]  
{ 
 \renewcommand{\fakeenv}{#2} 
 \theoremstyle{definition} 
 \newtheorem*{\fakeenv}{#1~\ref{#2}} 
 \begin{\fakeenv}
}
{
 \end{\fakeenv}
}

\newcommand{\cw}{\operatorname{cw}}
\newcommand{\tw}{\operatorname{tw}}
\newcommand{\Mid}{\operatorname{mid}}

\title[On the tree-width of knot diagrams]{On the tree-width of knot diagrams}

\author[De Mesmay]{Arnaud de Mesmay}
\address{Univ. Grenoble Alpes, CNRS, Grenoble INP, GIPSA-lab, 38000 Grenoble, France}
\email{arnaud.de-mesmay@gipsa-lab.fr}

\author[Purcell]{Jessica Purcell}
\address{School of Mathematics, 9 Rainforest Walk, Room 401, Monash University, VIC 3800, Australia}
\email{jessica.purcell@monash.edu}

\author[Schleimer]{Saul Schleimer}
\address{Mathematics Institute, University of Warwick, Coventry CV4 7AL, United Kingdom}
\email{s.schleimer@warwick.ac.uk}

\author[Sedgwick]{Eric Sedgwick}
\address{School of Computing, DePaul University, 243 S. Wabash Ave, Chicago, IL 60604, USA}
\email{esedgwick@cdm.depaul.edu}

\thanks{This work is in the public domain.}

\begin{document}

\begin{abstract}
We show that a small tree-decomposition of a knot diagram induces a small sphere-decomposition of the corresponding knot.  This, in turn, implies that the knot admits a small essential planar meridional surface or a small bridge sphere.  We use this to give the first examples of knots where any diagram has high tree-width.  This answers a question of Burton and of Makowsky and Mari\~no.
\end{abstract}

\maketitle

\section{Introduction}

The \emph{tree-width} of a graph is a parameter quantifying how ``close'' the graph is to being a tree.  While tree-width has its roots in structural graph theory, and in particular in the Robertson-Seymour theory of graph minors, it has become a key tool in algorithm design in the past decades.  This is because many algorithmic problems, on small tree-width graphs, can be solved efficiently using dynamic programming techniques.  

The algorithmic success of tree-width has also been demonstrated in the realm of topology.  Makowsky and Mari\~no showed that, despite being $\#P$-hard to compute in general, the Jones and Kauffman polynomials can be computed efficiently on knot diagrams where the underlying graphs have small tree-width~\cite{MakowskyMarino03}.  Recently Burton obtained a similar result for the HOMFLY-PT polynomial~\cite{Burton18}.  In a different vein, Bar-Natan used divide-and-conquer techniques to design an algorithm to compute the Khovanov homology of a knot; he conjectures that his algorithm is very efficient on graphs of low cut-width~\cite{BarNatan07}, a graph parameter closely related to tree-width.  Similarly, many topological invariants can be computed efficiently on manifold triangulations for which the face-pairing graphs have small tree-width; this is nicely captured by a generalization of Courcelle's theorem due to Burton and Rodney~\cite{BurtonDowney17}.

A \emph{diagram} $D$ of a knot $K$ is a four-valent graph embedded in $S^2$, the two-sphere, with some ``crossing'' information at the vertices.  Thus, we can apply the definition of tree-width directly to $D$.  The \emph{diagrammatic tree-width} of $K$ is then the minimum of the tree-widths of its diagrams.  Burton~\cite[page~2694]{BurtonEtAl15} and Makowsky and Mari\~no~\cite[page~755]{MakowskyMarino03} pose a natural question: is there a family of knots where the tree-width increases without bound? For example, while the usual diagrams of torus knots $T(p,q)$ (see Figure~\ref{F:torus}) can be easily seen to have tree-width going to infinity with $p$ and $q$, it could be that other diagrams of these torus knots have a uniformly bounded tree-width.

\begin{figure}[h]
  \centering
  \includegraphics[width=7cm]{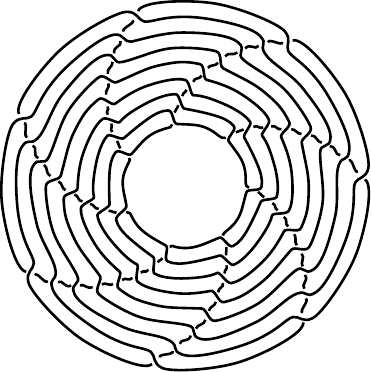}
  \caption{The usual diagram of a torus knot $T(9,7)$.}
  \label{F:torus}
  \end{figure}

In this article, we give several positive answers.  Our main tool comes from the analysis of surfaces in three-manifolds.  We defer the precise definitions to \refsec{Background}.

\begin{restate}{Theorem}{Thm:Main}
Suppose that $k$ is an integer and $K$ is a knot having a diagram with tree-width at most $k$.  Then either
\begin{enumerate}
\item there exists an essential planar meridional surface for $K$ with at most $8k + 8$ boundary components or
\item $K$ has bridge number at most $4k + 4$.
\end{enumerate}
\end{restate}

The main idea in the proof of \refthm{Main} is as follows.  Suppose that $D$ is the given small tree-width diagram of a knot $K$.  In \reflem{Sphere} we transform the given decomposition of $D$ into a family of disjoint spheres, each meeting $K$ in a small number of points.  This partitions the ambient space into simple pieces.  We call this a \emph{sphere-decomposition} of $K$.  In \reflem{Tubing} we turn this into a \emph{multiple Heegaard splitting}, as introduced by Hayashi and Shimokawa~\cite[page~303]{HayashiShimokawa01}.  Their thin position arguments, roughly following \cite{Gabai87iii}, establish the existence of the surfaces claimed by \refthm{Main}.

To find families of knots with diagrammatic tree-width going to infinity, we use the contrapositive of \refthm{Main}.  The following families of knots have neither small essential planar meridional surfaces nor small bridge spheres. 
\begin{itemize}
\item 
Torus knots $T(p,q)$ where $p$ and $q$ are large, coprime, and roughly equal.  These have no essential planar meridional surface by work of Tsau~\cite[page~199]{Tsau94}.  
Also, the bridge number is $\min (p, q)$, by work of Schubert~\cite[Satz~10]{Schubert54}.
\item 
Knots with a high distance Heegaard splitting, for example following Minsky, Moriah, and Schleimer~\cite[Theorem~3.1]{MinskyMoriahSchleimer07}. 
A theorem of Scharlemann~\cite[Theorem~3.1]{Scharlemann06} shows that such knots cannot have essential planar meridional surfaces in their complements with a small number of boundary components.  A theorem of Tomova~\cite[Theorem~1.3]{Tomova08} shows that such knots cannot have small bridge spheres.
\item 
Highly twisted plats, for example following Johnson and Moriah~\cite[Theorem~1.2]{JohnsonMoriah16}.  They show that the given bridge sphere has high distance.  A theorem of Bachman and Schleimer~\cite[Theorem~5.1]{BachmanSchleimer05a} excludes small essential planar meridional surfaces.  A theorem of Tomova~\cite[Theorem~10.3]{Tomova07} excludes bridge spheres with smaller bridge number than the given one. 
\end{itemize}

\begin{corollary}
For every integer $k > 0$, there is a knot $K$ with diagrammatic tree-width at least $k$.
\end{corollary}

\begin{proof}
By \refthm{Main}, the three families of knots above have diagrammatic tree-width going to infinity.
\end{proof}

\subsection{Relations with the tree-width of the complement}

Two recent works~\cite{HuszarSpreerWagner17, MariaPurcell18} have investigated the interplay between the tree-width of a three-manifold and its geometric or topological properties.  Here the tree-width of a three-manifold $M$ is the minimum of the tree-widths of the face-pairing graphs of its triangulations.  Interestingly, both articles find it more convenient to work with \emph{carving-width} (also called \emph{congestion}) than with tree-width directly.  This is also the case in our paper; see \refsec{Background} for definitions.

Husz\'{a}r, Spreer and Wagner~\cite[Theorem~2]{HuszarSpreerWagner17} give a result similar to ours, but for three-manifolds.  They prove that a certain family of manifolds $\{M_n\}$, constructed by Agol~\cite{Agol03}, has the property that any triangulation of $M_n$ has tree-width at least $n$.  Their result differs from ours in two significant ways.  For one, Agol's examples are non-Haken, while we work with knots and their complements, which are necessarily Haken.  Furthermore, there is an inequality between the diagrammatic tree-width of a knot and the tree-width of its complement which cannot be reversed.  That is, given a knot diagram of tree-width $k$, there are standard techniques to build a triangulation of its complement of tree-width $O(k)$.  For example, one may use the method embedded in SnapPy~\cite[\texttt{link\_complement.c}]{SnapPy}; see the discussion of~\cite[Remark~6.2]{MariaPurcell18}.

To see that this inequality cannot be reversed recall that \refthm{Main} gives a sequence of torus knots with tree-width going to
infinity.  On the other hand we have the following. 

\begin{restate}{Lemma}{Lem:Torus}
The complement of the torus knot $T(p,q)$ admits a triangulation of constant tree-width.
\end{restate}

\noindent
This is well-known to the experts; we include a proof in \refsec{Background}. 
There is much more to say about the tree-widths of triangulations of knot complements; however in this article we will restrict ourselves to knot diagrams. 

\subsection*{Acknowledgements}
We thank Ben Burton for his very provocative questions, and Kristof Husz{\'a}r and Jonathan Spreer 
for enlightening conversations.  We thank the Mathematisches Forschungsinstitut Oberwolfach and the organizers of Workshop 1542, as well as the Schlo\ss{} Dagstuhl Leibniz-Zentrum f\"ur Informatik and the organizers of Seminar 17072.  Parts of this work originated from those meetings.
The work of A. de Mesmay is partially supported by the French ANR project ANR-16-CE40-0009-01 (GATO) and the CNRS PEPS project COMP3D. Purcell was partially supported by the Australian Research Council.

\section{Background}
\label{Sec:Background}

\subsection{Graph Theory}

Unless otherwise indicated, all graphs will be simple and connected.
The tree-width of a graph measures quantitatively how close the graph is to a tree.  Although we provide a definition below for completeness, we will not use it in the proof of \refthm{Main}, as we rely instead on a roughly equivalent variant.

\begin{definition}
Suppose $G = (V, E)$ is a graph.  A \emph{tree-decomposition} of $G$ is a pair $(\calT, \calB)$ where $\calT$ is a tree, $\calB$ is a collection of subsets of $V$ called \emph{bags}, and the vertices of $\calT$, called \emph{nodes}, are the members of $\calB$.  Additionally, the following properties hold:

\begin{itemize}
\item $\bigcup_{B \in \calB} B = V$
\item For each $uv \in E$, there is a bag $B \in \calB$ with $u, v \in B$.
\item For each $u \in V$, the nodes containing $u$ induce a connected subtree of $\calT$.
\end{itemize}

\noindent
The \emph{width} of a tree-decomposition is the size of its largest bag, minus one.  The \emph{tree-width} of $G$, denoted by $\tw(G)$ is the minimum width taken over all possible tree-decompositions of $G$.
\end{definition}

Tree-width has many variants, which all turn out to be equivalent up to constant factors.  For our purposes, we will rely on the concept of carving-width; it has geometric properties that are well-suited to our setting. 

\begin{definition}
Suppose that $G$ is a graph.  A \emph{carving-decomposition} of $G$ is a pair $(\calT, \phi)$ where $\calT$ is a binary tree and $\phi$ is a bijection from the vertices of $G$ to the leaves of $\calT$.

For an edge $e$ of $\calT$, the \emph{middle set} of $e$, denoted $\Mid(e)$ is defined as follows.  Removing $e$ from $\calT$ breaks 
$\calT$ into two subtrees $S$ and $T$.  The leaves of $S$ and $T$ are mapped via $\phi$ to vertex sets $U$ and $V$.  Then $\Mid(e)$ is the set of edges connecting a vertex of $U$ to a vertex of $V$.

The \emph{width} of an edge $e$ is the size of $\Mid(e)$.  The width of the decomposition $(\calT, \phi)$ is the maximum of the widths of the edges of $\calT$.  Finally the \emph{carving-width} $\cw(G)$ is the minimum possible width of a carving-decomposition of $G$.
\end{definition}

The carving-width of a graph of constant degree is always within a constant factor of its tree-width, as follows. 

\begin{theoremQED}\cite[page~111 and Theorem~1]{Bienstock90}
\label{Thm:CW}
Suppose that $G$ is a graph with degree at most $d$.  Then we have: 
\[
\frac{2}{3} \cdot (\tw(G) + 1) \leq  \cw(G) \leq d \cdot (\tw(G) + 1). \qedhere
\]
\end{theoremQED}

Recall that a \emph{bridge} in a connected graph $G$ is an edge separating $G$ into more than one connected component.  A \emph{bond carving-decomposition} is one where, for any edge $e$ in the tree $\calT$, the associated vertex sets $U$ and $V$ induce connected subgraphs.  One of the strengths of the notion of a carving-decomposition is the following theorem of Seymour and Thomas (see also the discussion in Marx and Piliczuk~\cite[Section~4.6]{MarxPilipczuk15}).

\begin{theorem}\cite[Theorem~5.1]{SeymourThomas94}
\label{Thm:SeymourThomas}
Suppose that $G$ is a simple connected bridgeless graph with at least two vertices and with carving-width at most $k$.  Then there exists a bond carving-decomposition of $G$ having width at most $k$.  \qed
\end{theorem}

Suppose that $G$ has a bond carving-decomposition.  Suppose also that $G$ is \emph{planar}: it comes with an embedding into $S^2$.  Then each middle set $\Mid(e)$ for $G$ gives a Jordan curve $\gamma_e \subset S^2$ separating the two vertex sets $U$ and $V$.  One can take, for example, the simple cycle in the dual graph $G^*$ corresponding to the cut.  After a small homotopy of each, we can assume that these Jordan curves are pairwise disjoint.  We say that a family of pairwise disjoint Jordan curves, crossing $G$ transversely, \emph{realizes} a carving-decomposition $(\calT, \phi)$ if the partitions of the vertex sets that it induces correspond to those of $(\calT, \phi)$.  \refthm{SeymourThomas} and the above discussion yield the following.

\begin{corollary}
\label{Cor:Realize}
Suppose that $G$ is a bridgeless planar graph with at least two vertices and with carving-width at most $k$.  Then there exists a family of pairwise disjoint Jordan curves realizing a bond carving-decomposition of $G$ of width at most $k$. \qed
\end{corollary}

\subsection{Knots}
We refer to Rolfsen's book~\cite{Rolfsen90} for background on knot theory.  Here we will use the equatorial two-sphere $S^2$, and its containing three-sphere $S^3$, as the canvas for our knot diagrams.  In this way we avoid choosing a particular point at infinity.  If one so desires, they may place the point at infinity on the equatorial two-sphere and so obtain the $xy$--plane in $\RR^3$. 

A \emph{knot} $K$ is a regularly embedded circle $S^1$ inside of $S^3$.  One very standard way to represent $K$ is by a \emph{knot diagram}: we move $K$ to be generic with respect to geodesics from the north to south poles and we then project $K$ onto the equatorial two-sphere.  We add a label at each double point of the image recording which arc was closer to the north pole.  

Any knot has infinitely many knot diagrams, even when considering diagrams up to isotopy.  The four-valent graph underlying a knot diagram need not be simple, but it is always connected.  To avoid technical issues we may subdivide edges, introducing vertices of valence two, to ensure our graphs are simple.  Since knot diagrams are four-valent, this makes at most a constant difference in any relevant quantity. 

\begin{definition}
Suppose that $K$ is a knot with diagram $D$.  The \emph{tree-width} of $D$ is defined to be the tree-width of the underlying graph (after any subdivision needed to ensure simplicity).  We define the \emph{carving-width} of $D$ similarly.

The \emph{diagrammatic tree-width} (respectively \emph{diagrammatic carving-width}) of $K$ is defined to be the minimum tree-width (respectively carving-width) taken over all diagrams of $K$.
\end{definition}

\begin{remark}
It is relatively easy to find knots with small diagrammatic tree-width.  
For example, any connect sum of trefoil knots has uniformly bounded diagrammatic tree-width~\cite[Section~2.G]{Rolfsen90}. 
For example, any two-bridge knot has uniformly bounded diagrammatic tree-width~\cite[Section~4.D and Exercise~10.C.6]{Rolfsen90}.  
On the other hand, it is an exercise to show that any knot (even the unknot) has \emph{some} diagram of high tree-width.  
This then motivates the question of Burton~\cite[page~2694]{BurtonEtAl15} and of Makowsky and Mari\~no~\cite[page~755]{MakowskyMarino03}: is there a family of knots $\{K_n\}$ so that \emph{every} diagram of $K_n$ has tree-width at least $n$?
\end{remark}

\subsection{Tangles}

\begin{figure}[t]
  \centering
  \includegraphics[width=12cm]{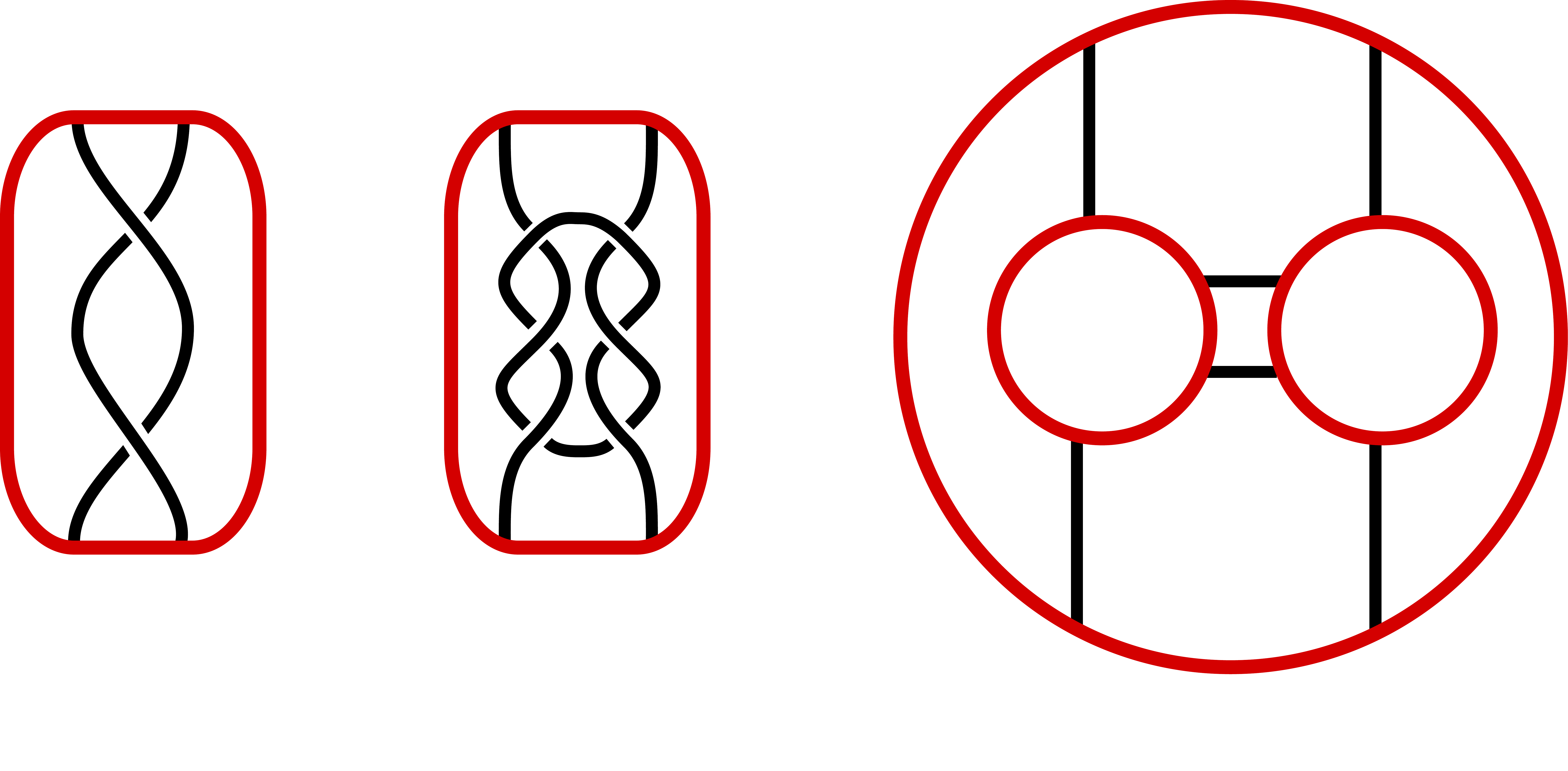}
  \caption{Three tangles represented by two-dimensional projections. The first one is trivial in a three-ball, the second one is non-trivial in a three-ball and the third one is flat in the solid pants.}
    \label{Fig:Tangles}
  \end{figure}

Suppose that $M$ is a compact, connected three-manifold.  A \emph{tangle} $T \subset M$ is a properly embedded finite collection of pairwise disjoint arcs and loops.  We call $(M, T)$ a \emph{three-manifold/tangle pair}, see \reffig{Tangles} for examples.  We call the components of $T$ the \emph{strands} of the tangle.  We take $N(T) \subset M$ to be a small, closed, tubular neighborhood of $T$.  Let $n(T)$ be the relative interior of $N(T)$.  So if $\alpha \subset T$ is a strand then $n(\alpha)$ is homeomorphic to an open disk crossed with a closed interval (or with a circle).  We define the \emph{tangle complement} to be $M_T = M - n(T)$.   Now suppose that $F \subset M$ is a properly embedded surface (two-manifold) which is transverse to $T$.  In particular, $T \cap \bdy F = \emptyset$.  Shrinking $n(T)$ if needed, we define $F_T = F - n(T)$.  Note that $F_T$ is properly embedded in $M_T$.

For any surface $F$, a simple closed properly embedded curve $\alpha \subset F$ is \emph{essential} on $F$ if it does not cut a disk off of $F$.  A properly embedded arc $\alpha \subset F$ is \emph{essential} on $F$ if it does not cut a \emph{bigon} off of $F$: a disk $D \subset F$ with $\bdy D = \alpha \cup \beta$ and with $\beta \subset \bdy F$. 

Suppose that $(M, T)$ is a three-manifold/tangle pair.  A \emph{meridian} for a tangle $T$ is an essential simple closed curve in $\bdy N(T) - \bdy M$ that bounds a properly embedded disk in $N(T)$.  The given disk is called a \emph{meridional disk}. A surface $F_T$ is \emph{meridional} if $\bdy F$ is empty and thus every component of $\partial F_T$ is a meridian of $T$. A surface is \emph{planar} if it is a subsurface of the two-sphere. For example, if $(B^3,T)$ is a three-ball/tangle pair, the boundary of the three-ball is a planar meridional surface. Similarly, a meridian for a solid torus $D \times S^1$ is an essential simple closed curve on the boundary $S^1 \times S^1$ bounding a properly embedded disk. For example, any disk of the form $D \cross \{\pt\}$, with product structure as above, is a meridional disk, see~\reffig{compDisks}, right. 

\begin{figure}[t]
  \centering
  \includegraphics[width=12cm]{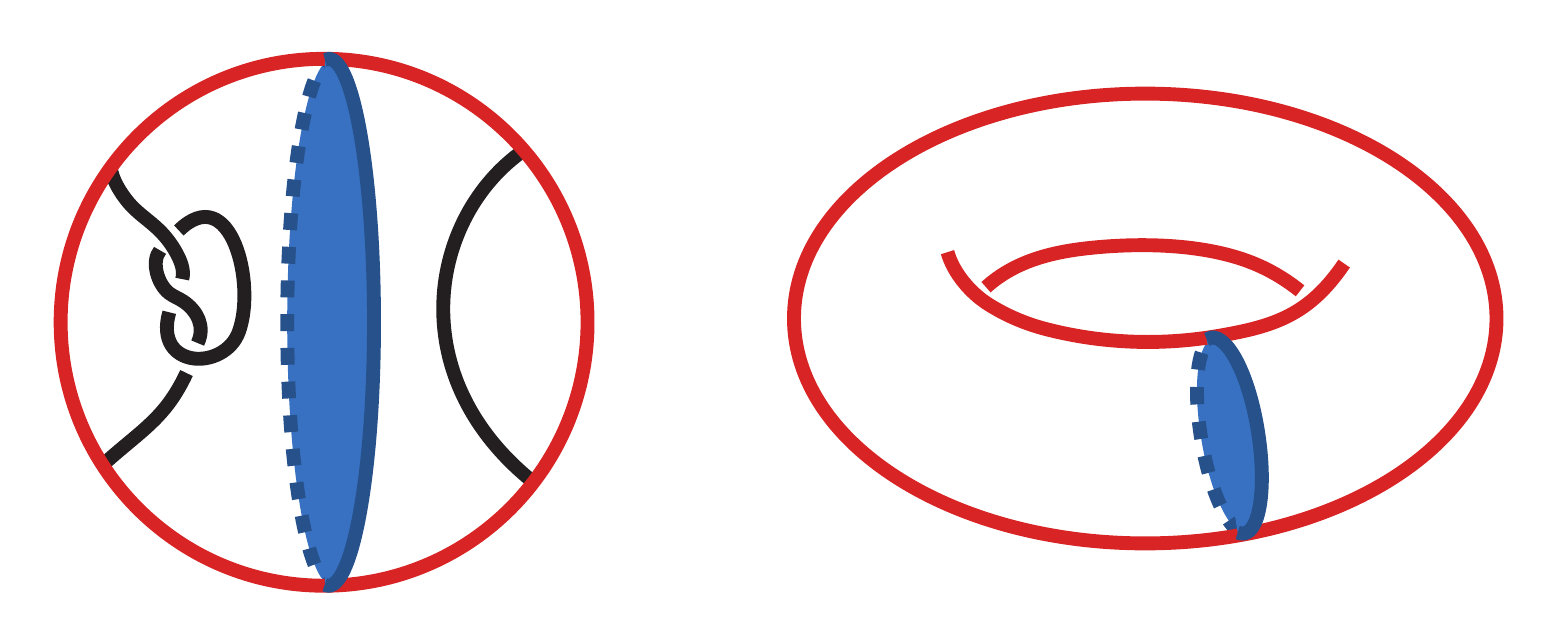}
  \caption{Left: a compressing disk for a planar meridional surface. Right: a meridional disk for the boundary of a solid torus.}
    \label{Fig:compDisks}
  \end{figure}

With the notion of meridians in hand, we can give the promised proof of \reflem{Torus}.  

\begin{lemma}
\label{Lem:Torus}
The complement of the torus knot $T(p,q) \subset S^3$ admits a triangulation of constant tree-width.
\end{lemma}

We prove this by building an explicit triangulation of the complement of torus knots which is very path-like.  We use Jaco and Rubinstein's layered triangulations~\cite{JacoRubinstein06}.

\begin{proof}[Proof of \reflem{Torus}]
Recall that the two-torus $T \homeo S^1 \cross S^1$ has a triangulation with exactly one vertex, three edges, and two triangles.  We call this the \emph{standard triangulation} of $T$; we label the edges by $a, b, c$.  We form $P = T \cross [-1, 1]$ and triangulate it with two triangular prisms, each made out of three tetrahedra.  Without changing the upper or lower boundaries ($T \cross \{\pm 1\}$) we subdivide the triangulation of $P$ to make $n(a \cross \{0\})$ an open subcomplex.   Finally, we form $Q = P - n(a \cross \{0\})$, together with its triangulation, by removing those open cells. 

Since the torus knot $T(p, q)$ is connected, the integers $p$ and $q$ are coprime. Using B\'ezout's identity, pick $u$ and $v$ such that $pv - qu = 1$.  Following Jaco and Rubinstein~\cite[Section~4]{JacoRubinstein06}, build two layered solid tori $U$ and $V$ of slopes $p/u$ and $q/v$: these are one-vertex triangulations of $D \cross S^1$ so that
\begin{itemize}
\item
the boundary of each is a standard triangulation of the two-torus, 
\item
the face-pairing graph for the triangulation of $U$ and $V$ is a daisy chain graph (see Figure~\ref{F:daisy}), and
\item
the meridian $\mu \subset \bdy U$ crosses the edges $(a, b, c)$ in $\bdy U$ exactly $(p, u, p + u)$ times while
\item
the meridian $\nu \subset \bdy V$ crosses the edges $(a, b, c)$ in $\bdy V$ exactly $(q, v, q + v)$ times.
\end{itemize}

\begin{figure}[t]
  \centering
  \includegraphics[width=8cm]{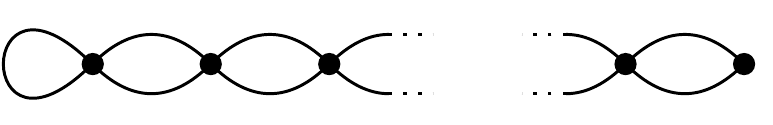}
  \caption{The daisy chain graph.}
  \label{F:daisy}
  \end{figure}

We now glue $U$ to $P$ by identifying $\bdy U$ with the upper boundary of $P$, respecting the $a$, $b$, $c$ labels.  Similarly we glue $V$ to $P$ along the lower boundary, again respecting labels.  Let $M = U \cup P \cup V$ be the result.  Since the upper and lower halves of $P$ are products, we may isotope $\mu$ down, and $\nu$ up, to lie in the middle torus $T \cross \{0\}$.  These curves now cross $|pv - qu| = 1$ times: that is, once.  We deduce that $M$ is homeomorphic to $S^3$; see for example Rolfsen~\cite[Section~9.B]{Rolfsen90}.  Now, the edge $a \cross \{0\} \subset P \subset M$ crosses these copies of $\mu$ and $\nu$ exactly $p$ and $q$ times, respectively.  So this edge gives a copy of $T(p,q)$ in $S^3$.   

We now form $X(p, q) = S^3 - n(T(p, q))$ by drilling $a \cross \{0\}$ out of $M$.  We do this by replacing the triangulation of $P$ by that of $Q$.  The resulting triangulation of $X(p, q)$ has a standard triangulation in $Q$ and a daisy chain triangulation in each of $U$ and $V$.   So the overall triangulation has constant tree-width, with the model tree being a path. 
\end{proof}

\subsection{Compression bodies}

Before defining sphere-decompositions of knots, we recall a few more notions from three-manifold topology. 

\begin{definition}
A \emph{one-handle} is a copy of the disk cross an interval, $D \cross [-1, 1]$, with \emph{attaching regions} $D \cross \{-1\}$ and $D \cross \{1\}$.  

Suppose that $F$ is a disjoint (and possibly empty) union of closed oriented surfaces.  We form a \emph{compression body} $C$ by starting with $F \cross [0, 1]$, taking the disjoint union with a three-ball $B$, and attaching one-handles to $(F \cross \{1\}) \cup \bdy B$.  We attach enough one-handles to ensure that $C$ is connected.  Then $\bdy_- C = F \cross \{0\}$ is the \emph{lower boundary} of $C$.  Also, $\bdy_+ C = \bdy C - \bdy_- C$ is the \emph{upper boundary} of $C$.  If $\bdy_- C = \emptyset$ then $C$ is a \emph{handlebody}. 
\end{definition}

Here is a concrete example which we will use repeatedly. 

\begin{definition}
Fix $n > 0$.  Suppose that $F$ is a disjoint union of $n - 1$ copies of $S^2$.  We thicken, take the disjoint union with a three-ball $B$, and attach $n - 1$ one-handles in such a way to obtain a compression body $C(n)$.  Note that $\bdy_+ C(n) \homeo S^2$.  We call $C(n)$ an \emph{$n$--holed three-sphere}.  Equally well, $C(n)$ has exactly $n$ boundary components, all spheres, and $C(n)$ embeds in $S^3$.  Equally well, $C(n)$ is obtained from $S^3$ by deleting $n$ small open balls with disjoint closures.
\end{definition}

The four simplest $n$-holed three-spheres are the three-sphere (by convention), the three-ball, the shell $S^2 \cross [0, 1]$, and the \emph{solid pants}: the thrice-holed three-sphere $C(3)$. 

Suppose that $B$ is the closed unit three-ball, centered at the origin of $\RR^3$.  Remove the open balls of radius $1/4$ centered at $(\pm 1/2, 0, 0)$, respectively to obtain the \emph{standard model} for $C(3)$.  The \emph{equatorial} pair of pants $P \subset C(3)$ is the intersection of $C(3)$ and the $xy$--plane.

\begin{definition}
Suppose that $C(3)$ is the standard model of the solid pants.  Suppose that $T \subset C(3)$ is a tangle.  We say $T$ is \emph{flat} if 
\begin{itemize}
\item
$T$ is properly embedded in $P \subset C(3)$, the equatorial pair of pants, 
\item
$T$ has no loop component, and
\item
all strands of $T$ are essential in $P$. 
\end{itemize}
\end{definition}


\begin{definition}
Suppose that $C$ is a compression body.  The construction of $C$ gives us a homeomorphism $C \homeo N \cup H$ where $N$ is either a three-ball or a copy of $\bdy_- C \times [0, 1]$ and where $H$ is a disjoint union of one-handles.  Suppose that $\alpha \subset C$ is a properly embedded arc.  
\begin{itemize}
\item
We call $\alpha$ a \emph{vertical} arc if $\alpha$ lies in $N$ and there has the form $\{\pt\} \cross [0, 1]$.
\item
We call $\alpha$ a \emph{bridge} if there is an embedded disk $D \subset C$ so that $\bdy D = \alpha \cup \beta$ with $\beta \subset \bdy_+ C$.  We call $D$ a \emph{bigon} for $\alpha$. 
\end{itemize}
A tangle $T \subset C$ is \emph{trivial} if 
\begin{itemize}
\item
every strand of $T$ is either vertical or is a bridge and 
\item
every bridge $\alpha \subset T$ has a bigon $D$ so that $D \cap T = \alpha$. 
\end{itemize}
\end{definition}



Note that when $C$ is a handlebody, a trivial tangle consists solely of bridges. We refer to~\reffig{Tangles} for examples.  Now we can begin to define the two types of surfaces that appear in \refthm{Main}. 

\begin{definition}
Suppose that $K$ is a knot in $S^3$.  A two-sphere $S \subset S^3$ is a \emph{bridge sphere} for $K$ if
\begin{itemize}
\item
$S$ is transverse to $K$ and 
 \item
the induced tangles in the two components of $S^3 - n(S)$ are trivial.  
\end{itemize}
The \emph{bridge number} of $K$ with respect to $S$ is the number of bridges in either trivial tangle. 
\end{definition}

\begin{definition}
Suppose that $M$ is a compact connected oriented three-manifold.  Suppose that $F \subset M$ is a properly embedded two-sided surface.  An embedded disk $(D, \bdy D) \subset (M, F)$ is a \emph{compressing disk} for $F$ if $D \cap \bdy M = \emptyset$, if $D \cap F = \bdy D$, and if $\bdy D$ is an essential loop in $F$, see~\reffig{compDisks} left for an illustration.  If $F$ does not admit any compressing disk, then $F$ is \emph{incompressible}.

A \emph{boundary compressing bigon} is an embedded disk $D \subset M$ with boundary $\bdy D = \alpha \cup \beta$ being the union of two arcs where $\alpha = D \cap F$ is an essential arc in $F$ and where $\beta = D \cap \bdy M$.   If $F$ does not admit any boundary compressing disk, then $F$ is \emph{boundary-incompressible}.

A closed surface $F$, embedded in the interior of $M$ is \emph{boundary parallel} if $F$ cuts a copy of $F \times [0, 1]$ off of $M$. 

A surface $F \subset M$ that is incompressible, is boundary-incompressible, and is not boundary-parallel is called \emph{essential}.
\end{definition}

\subsection{Multiple Heegaard splittings}

Here we recall the concept of a multiple Heegaard splitting, which is central to the work of Hayashi and Shimokawa~\cite{HayashiShimokawa01}.  We also state one of their theorems that we will rely on. 

Suppose that $M$ is a compact connected oriented three-manifold.  A \emph{Heegaard splitting} $\calC$ of $M$ is a decomposition of $M$ as a union of two compression bodies $C$ and $C'$, disjoint on their interiors, with $\bdy_+ C = \bdy_+ C'$ and with $\bdy M = \bdy_- C \cup \bdy_- C'$.  Following~\cite{ScharlemannSchultensSaito16, ScharlemannThompson94b}, a \emph{generalized Heegaard splitting} $\calC$ of $M$ is a path-like version of a Heegaard splitting; it is a decomposition of $M$ into a sequence of compression bodies $C_i$ and $C_i'$, disjoint on their interiors, where $\bdy_+ C_i = \bdy_+ C_i'$, where $\bdy_- C_i' = \bdy_- C_{i+1}$, and where $\bdy M = \bdy_- C_1 \cup \bdy_- C_n'$.

Following Hayashi and Shimokawa~\cite[page~303]{HayashiShimokawa01} we now generalize generalized Heegaard splittings.  We model the decomposition of $M$ on a graph instead of just a path.  We also must respect a given tangle $T$ in $M$.  Here is the definition. 

\begin{definition}
Suppose that $(M, T)$ is a three-manifold/tangle pair.  A \emph{multiple Heegaard splitting} $\calC$ of $(M, T)$ is a decomposition of $M$ into a union of compression bodies $\{C_i\}$, with the following properties. 
\begin{enumerate}
\item For each index $i$ there is some $j \neq i$ so that $\partial_+ C_i$ is identified with $\partial_+ C_j$.
\item For each index $i$ and for each component $F \subset \partial_- C_i$ either $F$ is a component of $\bdy M$ or there is an index $j$ so that $F$ is attached to some component $G \subset \partial_- C_j$.  Here we allow $j = i$ but require $G \neq F$.  
\item The surfaces $\cup \bdy_\pm C_i$ are transverse to $T$.  The tangle $T \cap C_i$ is trivial in $C_i$. 
\end{enumerate}
\end{definition}

When discussing the union of surfaces we will use the notation $\bdy_\pm \calC = \cup \bdy_\pm C_i$.  Again, following~\cite[page~303]{HayashiShimokawa01} we give a complexity of multiple Heegaard splittings. 

\begin{definition}
Suppose that $(M, T)$ is a three-manifold/tangle pair.  Suppose that $F \subset M$ is a connected closed surface embedded in the interior of $M$.  Suppose also that $F$ is transverse to $T$.  The \emph{complexity} of $F$ is the ordered pair $c(F) = (\text{genus(F)}, |F \cap T|)$.  

Now, suppose that $\calC$ is a multiple Heegaard splitting of $(M, T)$.
\begin{itemize}
\item 
The \emph{cost} of $\calC$ is the maximal number of intersections between a component of $\bdy_+ \calC$ and $T$.  That is, the cost is
\[
\max\{ |F \cap T| : \mbox{$F$ is a component of $\bdy_+ \calC$} \}. 
\]
\item 
The \emph{width} of $\calC$ is the ordered multiset of complexities
\[ 
w(\calC) = \{ c(F) : \mbox{$F$ is a component of $\bdy_+ \calC$} \}.
\] 
Here the complexities are listed in lexicographically non-increasing order.
\end{itemize}

The \emph{width} of a pair $(M,T)$ is the minimum possible width of a multiple Heegaard splitting of $(M,T)$, ordered lexicographically. 

A multiple Heegaard splitting of $(M,T)$ is \emph{thin} if it achieves the minimal width over all possible multiple Heegaard splittings of $(M,T)$.
\end{definition}

The following definitions, still from Hayashi and Shimokawa~\cite[page~304]{HayashiShimokawa01}, provide a notion of essential surfaces in the setting of a three-manifold/tangle pair. Recall that if $X \subset M$ is transverse to $T$ then $X_T$ denotes $X - n(T)$.

\begin{definition}
  Let $X$ be a compact orientable $3$-manifold, $T$ a $1$-manifold properly embedded in $X$, and $F$ a closed orientable $2$-manifold embedded transversely to $T$ in $X$. Let $\tilde{X}$ be the $3$-manifold obtained from $X$ by capping off all the spherical boundary components disjoint from $T$ with balls. An embedded disk $Q$ is said to be a \emph{thinning disk} of $F$ if $T \cap Q = T \cap \partial Q= \alpha$ is an arc and $Q \cap F$ contains the arc $cl(\partial Q \setminus \alpha)=\beta$ as a connected component. The surface $F$ is \emph{$T$--essential} if
  \begin{enumerate}
  \item $F_T$ is incompressible in $X_T$,
  \item $F$ has no thinning disk,
  \item no component $F_0$ of $F$ cobounds with a component $F_1$ of $\partial X$ in $\tilde{X}$ a $3$-manifold homeomorphic to $F_0 \times I$, possibly intersecting $T$ in vertical arcs, where $F_0=F_0 \times \{0\}$ and $F_0=F_1 \times \{1\}$, and
  \item no sphere component of $F$ bounds a ball disjoint from $T$ in $\tilde{X}$.
\end{enumerate}
  \end{definition}

Note that when $\partial X =\emptyset$, which will be the case throughout this paper, a $T$--essential surface $F$ is incompressible and boundary-incompressible in $X_T$. We can now state the necessary result from~\cite{HayashiShimokawa01}.

\begin{theorem}
\label{Thm:HS}
Suppose that $\calC$ is a thin multiple Heegaard splitting of a three-manifold/tangle pair $(M,T)$.  Then every component of $\bdy_- \calC$ is $T$--essential in $(M,T)$.
\end{theorem}

\begin{proof}
This follows from Theorem~1.1 and Lemma~2.3 of~\cite{HayashiShimokawa01}.
\end{proof}


\section{Sphere-decompositions}

We connect multiple Heegaard splittings to tree-width by turning the Jordan curves of \refcor{Realize} into spheres, and broadening these into (almost) a multiple Heegaard splitting.

\begin{definition}
Let $K$ be a knot in $S^3$.  A \emph{sphere-decomposition} $\calS$ of $K$ is a finite collection of pairwise disjoint, embedded spheres in $S^3$ meeting $K$ transversely, and so that every component $X$ of $S^3 - n(\calS)$ is either
\begin{itemize}
\item
a three-ball, and $(X, X \cap K)$ is a trivial tangle, or
\item
a solid pants, and $(X, X \cap K)$ is a flat tangle.
\end{itemize}
\end{definition}


We now define various notions of complexity of a sphere-decomposition. 

\begin{definition}
The \emph{weight} of a sphere $S$ in a sphere-decomposition $\calS$ is the number of intersections $S \cap K$.  The \emph{cost} of $\calS$ is the weight of its heaviest sphere.  The \emph{width} of $\calS$ is the list of the weights of its spheres, with multiplicity, given in non-increasing order.  
\end{definition}

As an example, a bridge sphere $S$ for a knot in $S^3$ gives a sphere-decomposition with one sphere; the width is $\{2b\}$, where $b$ is number of bridges of $K$ on either side of $S$. 

As another example consider a pretzel knot, as in \reffig{Pretzel}, or even more generally a Montesinos knot consisting of three
rational tangles.  Then the three two-spheres about the three rational tangles give a sphere-decomposition of width $\{4, 4, 4\}$.  

\begin{figure}
  \centering
  \def\svgwidth{5cm}
\begingroup%
  \makeatletter%
  \providecommand\color[2][]{%
    \errmessage{(Inkscape) Color is used for the text in Inkscape, but the package 'color.sty' is not loaded}%
    \renewcommand\color[2][]{}%
  }%
  \providecommand\transparent[1]{%
    \errmessage{(Inkscape) Transparency is used (non-zero) for the text in Inkscape, but the package 'transparent.sty' is not loaded}%
    \renewcommand\transparent[1]{}%
  }%
  \providecommand\rotatebox[2]{#2}%
  \newcommand*\fsize{\dimexpr\f@size pt\relax}%
  \newcommand*\lineheight[1]{\fontsize{\fsize}{#1\fsize}\selectfont}%
  \ifx\svgwidth\undefined%
    \setlength{\unitlength}{212.598bp}%
    \ifx\svgscale\undefined%
      \relax%
    \else%
      \setlength{\unitlength}{\unitlength * \real{\svgscale}}%
    \fi%
  \else%
    \setlength{\unitlength}{\svgwidth}%
  \fi%
  \global\let\svgwidth\undefined%
  \global\let\svgscale\undefined%
  \makeatother%
  \begin{picture}(1,1.10000212)%
    \lineheight{1}%
    \setlength\tabcolsep{0pt}%
    \put(0,0){\includegraphics[width=\unitlength,page=1]{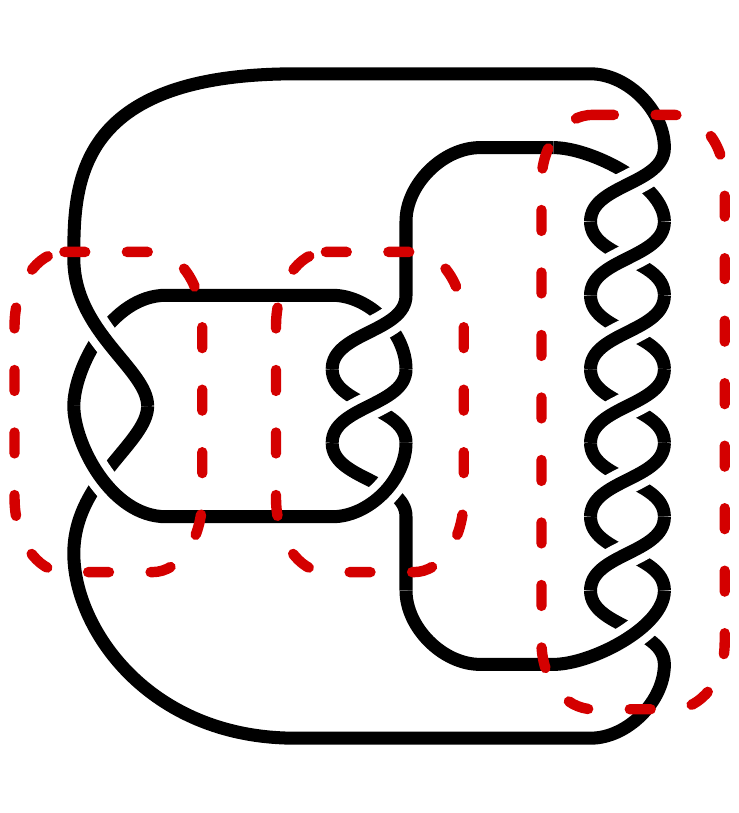}}%
    \put(0.22384061,0.21557242){\color[rgb]{0,0,0}\makebox(0,0)[t]{\lineheight{1.25}\smash{\begin{tabular}[t]{c}$S_1$\end{tabular}}}}%
    \put(0.4592224,0.21557242){\color[rgb]{0,0,0}\makebox(0,0)[t]{\lineheight{1.25}\smash{\begin{tabular}[t]{c}$S_2$\end{tabular}}}}%
    \put(0.90901147,0.01631351){\color[rgb]{0,0,0}\makebox(0,0)[t]{\lineheight{1.25}\smash{\begin{tabular}[t]{c}$S_3$\end{tabular}}}}%
  \end{picture}%
\endgroup%

\caption{The pretzel knot $P(-2,3,7)$ with a sphere-decomposition made of three spheres $S_1$, $S_2$ and $S_3$\protect\footnotemark.}
\label{Fig:Pretzel}
\end{figure}
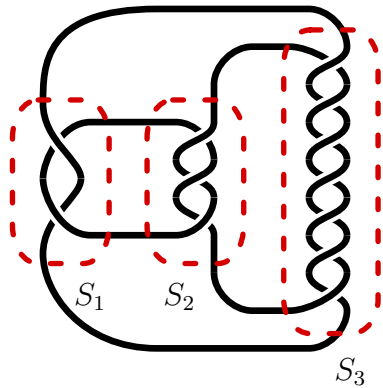


A carving-decomposition of a knot diagram yields a sphere-decomposition of the knot, as follows.

\begin{lemma}
\label{Lem:Sphere}
Suppose that $K$ is a knot. Suppose that $D$ is a diagram of $K$ with carving-width at most $k$.  Then there is a sphere-decomposition of $K$ with cost at most $k$.
\end{lemma}

\begin{proof}
Let $S$ be the equatorial sphere containing the graph $D$.  Note that $D$ is bridgeless.  Thus, by \refcor{Realize}, there is a family $\Gamma$ of pairwise disjoint Jordan curves in $S$ realizing a minimal width carving-decomposition of $D$.  By assumption this has width at most $k$.

We obtain the knot $K$ by moving the overstrands of $D$ slightly towards the north pole of $S^3$ while moving the understrands slightly towards the south pole.  Every Jordan curve $\gamma \in \Gamma$ can be realized as the intersection of a two-sphere $S(\gamma)$ with the equatorial sphere $S$: we simply cap $\gamma$ off with disks above and below $S$.   We do this in such a way that resulting spheres $S(\gamma)$ are pairwise disjoint and only meet $K$ near $S$. 

We claim that the resulting family of spheres $\calS = \{S(\gamma)\}$ is a sphere-decomposition of $K$ of cost at most $k$.  First, the bound on the cost follows from the bound on the width of the initial carving-decomposition realized by $\Gamma$.  The arborescent and tri-valent structure of the carving-decomposition translates into the fact that every sphere in $\calS$ is adjacent, on each side, to either a pair of spheres (respectively to no other sphere) as the corresponding half-edge is not (respectively is) a leaf of the carving-tree.  Since we are in $S^3$, this implies that the connected components of $S^3 - n(\calS)$ are either three-balls (at the leaves) or solid pants.  At the leaves, the tangle in the three-ball is a small neighborhood of a single crossing in the original diagram $D$ (or of a valence two vertex).  Thus the tangles at the leaves are trivial.  Inside each solid pants $C$, the tangle lies in the equatorial pair of pants $P$.  This is because there are no crossings of $D$ in $P$.  Finally, the tangle is flat in $P$ as otherwise we could find a carving of lower width.
\end{proof}

\footnotetext{This picture and the next one are adapted from a public domain \href{https://commons.wikimedia.org/wiki/File:Pretzel_knot.svg}{figure} from the Wikimedia commons by Sakurambo.}

A sphere-decomposition is not in general a multiple Heegaard splitting, as the example of \reffig{Pretzel} shows.  However, the following lemma shows that a sphere-decomposition of small cost can be upgraded to a multiple Heegaard splitting of small cost.

\begin{lemma}
\label{Lem:Tubing}
Suppose that $K \subset S^3$ is a knot, and suppose that $K$ admits a sphere-decomposition $\calS$ of cost at most $k$.  Then $(S^3, K)$ admits a multiple Heegaard splitting $\calC$ of cost at most $2k$, where all components of $\bdy_\pm \calC$ are spheres.
\end{lemma}

\begin{proof}
The construction simply adds a ``thick'' sphere into every component of $S^3 - n(\calS)$.  Let $C$ be a solid pants of $S^3 - n(\calS)$; label its boundaries $U$, $V$, and $W$ and set $u = |K \cap U|$ and so on.  Suppose that $\alpha$ is a strand of $K \cap C$ so that, relabelling as needed, $\alpha$ meets $U$ and $V$.  We form a two-sphere $S(C)$ by tubing $U$ to $V$ using an unknotted tube parallel to the strand $\alpha$ but such that $\alpha$ is outside of the tube.  See \reffig{Pretzel2} for an example. 

\begin{figure}
  \centering
  \def\svgwidth{5cm}
\begingroup%
  \makeatletter%
  \providecommand\color[2][]{%
    \errmessage{(Inkscape) Color is used for the text in Inkscape, but the package 'color.sty' is not loaded}%
    \renewcommand\color[2][]{}%
  }%
  \providecommand\transparent[1]{%
    \errmessage{(Inkscape) Transparency is used (non-zero) for the text in Inkscape, but the package 'transparent.sty' is not loaded}%
    \renewcommand\transparent[1]{}%
  }%
  \providecommand\rotatebox[2]{#2}%
  \newcommand*\fsize{\dimexpr\f@size pt\relax}%
  \newcommand*\lineheight[1]{\fontsize{\fsize}{#1\fsize}\selectfont}%
  \ifx\svgwidth\undefined%
    \setlength{\unitlength}{216.55093073bp}%
    \ifx\svgscale\undefined%
      \relax%
    \else%
      \setlength{\unitlength}{\unitlength * \real{\svgscale}}%
    \fi%
  \else%
    \setlength{\unitlength}{\svgwidth}%
  \fi%
  \global\let\svgwidth\undefined%
  \global\let\svgscale\undefined%
  \makeatother%
  \begin{picture}(1,1.07992263)%
    \lineheight{1}%
    \setlength\tabcolsep{0pt}%
    \put(0,0){\includegraphics[width=\unitlength,page=1]{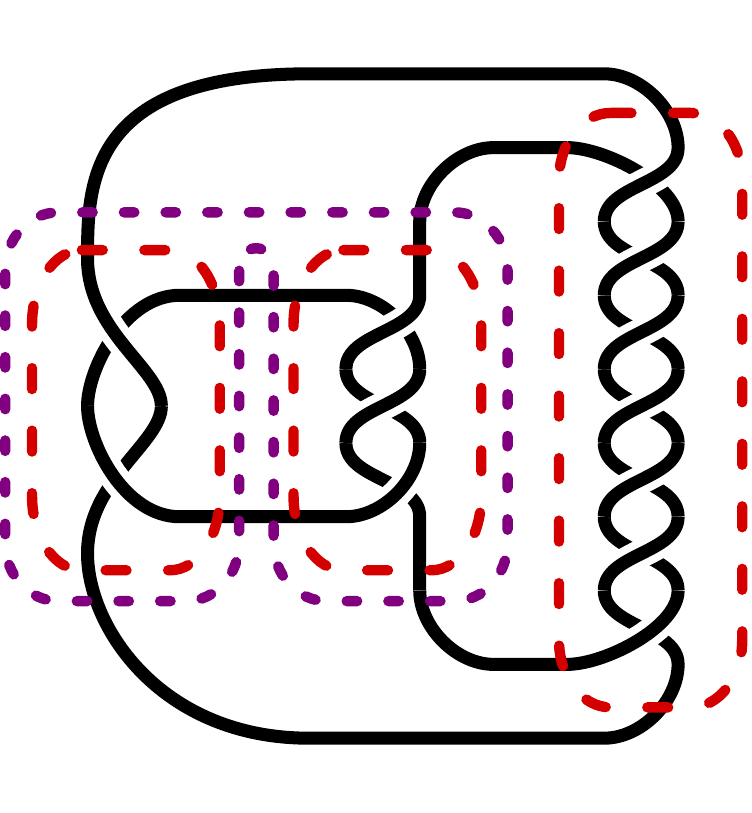}}%
    \put(0.34178434,0.86847476){\color[rgb]{0,0,0}\makebox(0,0)[t]{\lineheight{1.25}\smash{\begin{tabular}[t]{c}$S'$\end{tabular}}}}%
  \end{picture}%
\endgroup%

\caption{The pretzel knot $P(-2,3,7)$ with its natural sphere-decomposition, and with the addition of a ``thick'' sphere $S'$ in the solid-pants component, obtained by tubing $S_1$ to $S_2$.  The spheres of $\calS$ are dashed and the thick sphere is dotted (compare with \reffig{Pretzel}).}
\label{Fig:Pretzel2}
\end{figure}

The sphere $S(C)$ meets $K$ in $u + v$ points.  By the definition of the cost, we have $\max \{ u, v, w \} \leq k$, thus $u + v \leq 2k$.  We perform this tubing construction in every solid pants $C$ of $S^3 - n(\calS)$; we place the sphere $S(C)$ into the set $\calR$.  

For every three-ball component $B$ of $S^3 - n(\calS)$ we take $S(B)$ to be a push-off of $\bdy B$ into $B$.  We place the sphere $S(B)$ into the set $\calR$.  We claim that $\calC = S^3 - n(\calR \cup \calS)$ is a multiple Heegaard splitting, where $\bdy_+ \calC = \calR$ and where $\bdy_- \calC = \calS$.  Note that every component $C' \subset \calC$ is either a three-ball, a shell, or a solid pants.  So there is a unique (up to isotopy) compression body structure on $C'$ that has $\bdy_+ C' \subset \calR$ and $\bdy_- C' \subset \calS$.  Let $C$ be the component of $S^3 - n(\calS)$ containing $C'$. 
\begin{itemize}
\item
Suppose that $C'$ is a three-ball.  Then $K \cap C'$ is either one or two boundary parallel bridges, by the construction of $\calS$. 
\item
Suppose that $C'$ is a shell and $C$ is a three-ball.  Then $K \cap C'$ is four vertical arcs, because $\bdy_+ C'$ is parallel to $\bdy C$. 
\item
Suppose that $C'$ is a shell and $C$ is a solid pants.  We use the notation of the construction: $\bdy C = U \cup V \cup W$ and $\bdy_+ C'$ is obtained by tubing $U$ to $V$.  The tangle $K \cap C'$ consists of vertical arcs (coming from strands of $K \cap C$ connecting $U$ or $V$ to $W$) and bridges parallel into the upper boundary (coming from strands of $K \cap C$ connecting $U$ to $V$, or $U$ to itself, or $V$ to itself). 
\item
Suppose that $C'$ is a solid pants and thus $C$ is a solid pants.  Then $K \cap C'$, by construction, consists of vertical arcs only.
\end{itemize}
This concludes the proof.
\end{proof}

We are now ready to prove our main result. 

\begin{theorem}
\label{Thm:Main}
Suppose that $k$ is an integer and $K$ is a knot having a diagram with tree-width at most $k$.  Then either
\begin{enumerate}
\item there exists an essential planar meridional surface for $K$ with at most $8k + 8$ boundary components or
\item $K$ has bridge number at most $4k + 4$.
\end{enumerate}
\end{theorem}


\begin{proof}
Let $D$ be a diagram  of $K$ that has tree-width at most $k$.  Then by \refthm{CW}, it has carving-width at most $4k + 4$.  Applying successively Lemmas~\ref{Lem:Sphere} and~\ref{Lem:Tubing} to this carving-decomposition, we obtain a multiple Heegaard splitting $\calC$ of $(S^3, K)$ of cost at most $8k + 8$ and in which all the surfaces are spheres.

Now let us consider a thin multiple Heegaard splitting $\calC'$ of $(S^3, K)$.  Since $\calC'$ is thin, its width is less than or equal to that of $\calC$.  Thus all of the surfaces of $\bdy_+ C'$ must be spheres, as genus is the first item in our measure of complexity.  We next deduce that the cost of $\calC'$ is at most that of $\calC$, so is at most $8k + 8$.

We now have two cases, depending on whether or not $\bdy_- \calC'$ is empty.  If it is non-empty, then, taking $T=K$ in \refthm{HS} tells us that every component of $\bdy_- \calC'$ is $K$--essential in $S^3$. Thus every component of $\bdy_- \calC'_K$ is incompressible and boundary-incompressible in $S^3_K$. Furthermore, since all the surfaces of $\bdy_+ C'$ are spheres, the components of $\bdy_- \calC'_K$ are not boundary-parallel, therefore they are essential. Since the cost of $\calC'$ is at most $8k + 8$, we deduce that $K$ admits an essential planar meridional surface with at most $8k + 8$ boundary components. 

Suppose instead that $\bdy_- \calC'$ is empty.  Thus $\bdy_+ \calC'$ consists of a single sphere $S$, which is necessarily a bridge sphere.  Again, the cost of $\calC'$ is at most $8k + 8$, so $S$ has at most $4k + 4$ bridges.  This concludes the proof.
\end{proof}

\bibliographystyle{hyperplain} 
\bibliography{bibfile}
\end{document}